\documentclass[12 pt]{article}

\usepackage{amssymb,mathrsfs}
\usepackage{latexsym}
\usepackage{amsfonts}
\usepackage{amsmath}
\usepackage{amsthm}
\theoremstyle{plain}
\newtheorem{theorem}{Theorem}
\newtheorem{lemma}[theorem]{Lemma}
\newtheorem{proposition}[theorem]{Proposition}

\newtheorem{prob}{Problem}

\newcommand{\Z}{\mathbb{Z}}

\newcommand{\Q}{\mathbb{Q}}
\newcommand{\C}{\mathbb{C}}
\newcommand{\PP}{\mathbb{P}}
\newcommand{\kk}{\mathbf{k}}

\newcommand{\SSS}{\mathcal{S}}

\newcommand{\rk}{\mbox{\upshape rk}\,}

\begin{document}

\title{Quotients of Spheres By Linear Actions of Tori}
\author{Marisa J. Hughes\thanks{Research partially supported by NSF grant DMS-0900912} and Ed Swartz\thanks{Research partially supported by NSF grant DMS-0900912}\\
Department of Mathematics,\\
Cornell University, Ithaca NY, 14853-4201, USA,\\
mjbelk@math.cornell.edu}

\maketitle

\begin{abstract}

\indent We consider  quotients of spheres by linear actions of real tori. To each quotient we  associate a matroid  built out of a diagonalization of the torus action.  We find the integral homology groups of the resulting quotient spaces  in terms of the Tutte polynomial of the  matroid.  We also find the homotopy type and homology of the singular space of such an action.  Lastly, we consider the circumstances under which the orbit space is a manifold or, more specifically, a (homology) sphere.
\end{abstract}

\section{Introduction}

Let $G$ be a compact group that acts by isometries on a Riemannian manifold $Y$.  It is natural to ask whether the orbit space of this action is itself a topological manifold.  In order to answer this question, we can consider the behavior of the action on the tangent space.  Let $x$ be a point of $Y$ with tangent space $T_x Y$.  Denote by $S_x$ the unit tangent vectors in $T_x M$ which are perpendicular to the orbit $Gx$.  The isotropy group of $x$, $G_x = \{g \in G: gx=x\},$  acts on $S_x$.  In order for the overall orbit space $Y/G$ to be a manifold, the quotient $S_x/G_x$ must  at least be a homology sphere for each point $x \in Y$.  Thus, understanding the topology of quotients $S^n/G$ where $G \subseteq O(n)$ is essential for answering questions about the general orbit space $Y/G$.

 The geometry and topology of quotients of spheres by finite tori (i.e. groups of the form $(\mathbb{Z}_p)^r, p$ a prime) was investigated in \cite{EdThesis}.  Here we use similar techniques to analyze the quotients of spheres by linear actions of real tori $X = S^{2n-1}/T^r.$  In order to compute $H_\ast(X)$ we rewrite the action as a matrix that describes how each circle in the product $T^r$ acts on the invariant circles of $S^{2n-1}.$  The matroid represented by this matrix determines the Poincar\'e polynomial of the quotient via  the Tutte Polynomial.

We then move on to study the singular set of the quotient as an arrangement.  The lattice of the singular set arrangement is shown to correspond to the lattice of flats of the matroid described above.  The homotopy type of the singular set can then be computed using  techniques  described in \cite{ZZTop} and our previous results.  This allows a formulation of the Poincar\'e polynomial of the singular set in terms of the Tutte polynomial.

Having a formula for the Poincar\'e polynomial of $S^{2n-1}/T^r$ in terms of the Tutte polynomial gives us the tools necessary to determine when these orbit spaces are manifolds.   We classify all the actions of $T^r$ whose orbit space is a manifold, and, even more specifically, when  the orbit space is a (homology) sphere.
\section{Background and Notation}

\subsection{Matroids}

For a more thorough introduction to the theory of matroids, and proofs of the facts given below, see \cite{Oxley}.

\indent A \emph{matroid} is a pair $(E,I)$ where $E$ is a finite set, $\mathcal{P}(E)$ the power set of $E,$ and $\mathcal{I} \subseteq \mathcal{P}(E)$. The finite set $E$ is known as the {\em ground set}, and $\mathcal{I}$ is the set of {\em independent} subsets of $E$.  In order to be a matroid, the independent sets must respect the following axioms:

\noindent I1) $\emptyset \in \mathcal{I}$\\
I2) If $I_1 \in \mathcal{I}$ and $I_2 \subseteq I_1$, then $I_2 \in \mathcal{I}.$
\\
I3) If $I_1, I_2, \in \mathcal{I}$ and $|I_1|< |I_2|$, then $\exists   x\in I_2 \backslash I_1$ such that $I_1 \cup \{x\} \in \mathcal{I}$

An element $e\in E$ is called a \emph{loop} if it is contained in no independent sets.  We say $e \in E$ is a \emph{coloop} if it is contained in every maximal independent set of $M$.  In this paper, we will only be concerned with \textit{representable} matroids.  These matroids can be represented by a matrix; the ground set is the set of column vectors of a matrix, and the independent sets are precisely the sets of columns which are linearly independent as column vectors.  All of our matrices will have $\mathbb{Z}$ coefficients.  Linear independence of columns will always be over the rationals.     Row operations, which preserve the linear independence relations of the columns,   and column switches of a matrix do not change the isomorphism class of the matroid.

The \emph{deletion} of a matroid element, denoted $M-e$, has ground set $E \setminus e$ and independent sets $\mathcal{I}_{M-e} = \{I \in \mathcal{I}_M : e \not \in I\}$.  Deleting $e$ from a representable matroid can be accomplished by deleting the column corresponding to $e$ in a representative matrix.

 Another matroid construction, denoted by $M/e$,  is  the {\em contraction} of $M$ by $e$.  If $e$ is a loop of $M,$ then $M/e$ is the same as $M-e.$  Otherwise, $M/e$ has ground set $E\setminus e$ and independent sets $\mathcal{I}_{M/e} = \{ I \setminus e \:  : \: \{e\} \cup I \in \mathcal{I}_M \}$.  In the case of a representable matroid, the contraction by $e$ can be computed by row reducing a representative matrix $ A$ such that the column corresponding to $e$ has only one nonzero entry.  By deleting the row where this entry is located, along with the column corresponding to $e$, we get a new matrix that represents the contraction $M/e$.  For a subset $A$ of $E$ the contraction $M/A$ is obtained by contracting each element of $A$ one at a time.  It is not hard to show that $M/A$ is independent of the order in which the contractions are performed.  

A matroid is a \emph{direct sum} of matroids, $M = M_1 \oplus M_2,$ if the ground set of $M$ is the disjoint union of the ground sets $E_1$ and $E_2$, and a set $A$ is independent in $M$ if and only if $A \cap E_1 \in \mathcal{I}_1$ and $A\cap E_2 \in \mathcal{I}_2$.  Note that for any matroid $M$ with loop or coloop $e$, $M$ can be decomposed at $(M-e) \oplus e$.

Every matroid has a rank function $r: E \rightarrow \mathbb{N}_0$ that maps a set to the cardinality of its maximally independent subsets.  A \emph{flat}, or closed set, of a matroid is a subset $F \subseteq E$ such that $ \forall e\in E-F ,\: r(F) = r(F \cup e) -1$.    A {\em hyperplane} $H$ of a matroid is a flat of $M$ such that $r(H) = r(E)-1$. We will frequently use $r(M)$, or just $r,$  for  $r(E).$

\subsection{Lattice of flats}

 The flats of $M$ form a \emph{lattice} under inclusion which we will denote $L_M$.  A lattice is a partially ordered set in which each pair of elements has a unique least upper bound and greatest lower bound.  The lattice of flats of any matroid is \emph{coatomic}, i.e. any flat of $M$ can be realized as an intersection of hyperplanes.  If $F$ is a flat of $M$, then the interval $[F,E]=\{ F' \in L_M: F \subseteq F' \subseteq E\}$ is isomorphic as a poset to $L_{M/F}.$  

For any finite poset $P$ the {\em order complex} of $P,$ denoted $\Delta(P),$ is the simplicial complex whose vertices are the elements of $P$ and whose faces are chains in $P.$  Let $\widetilde{L_M}$ be $L_M$ with its least element, the subset of all loops, and greatest element, $E,$ removed.  The homotopy type of $\Delta(\widetilde{L_M})$ plays a key role in Section \ref{singset}.   

\begin{theorem} \cite{Bjorner} \label{ordercomplex}
The order complex $\Delta(\widetilde{L_M})$ is homotopy equivalent to a wedge of $\mu(M)$ spheres all of which have dimension $r(M)-2.$\end{theorem}

\subsection{Tutte Invariants} \label{tutte invariants}

\indent The \textit{Tutte Polynomial}, written $T(M;x,y)$, is a matroid invariant that behaves well with respect to deletion and contraction.  It is defined as the unique two-variable polynomial satisfying the following recursion:\\
1) $T($a single coloop$; x, y) = x; \:\: T($a single loop$; x,y) = y $\\
2) If $e$ is a loop or a coloop, then $T(M; x,y) = T(e; x,y) T(M/e; x,y)$\\
3) If $e$ is neither a loop nor a coloop, then $T(M; x,y) = T(M-e; x,y) + T(M/e; x,y)$

\noindent It is sometimes preferrable to replace (2) with the following:  If $M = M_1\oplus M_2$, then $T(M; x,y) = T(M_1; x,y) T(M_2; x,y)$.  This definition is equivalent.

The Tutte polynomial is well-defined and unique for any matroid.   See \cite{TutteP} for a proof and many more applications of this polynomial.

The M\"{o}bius function of a finite poset is the function $\mu : L \times L \rightarrow \mathbb{Z}$ that satisfies:\\
$\forall x,y,z \in L,  \displaystyle \sum_{x \leq y \leq z} \mu(x, z) = \delta(x,z)$ and $\mu(x,z) = 0$ if $x \not\leq z$.  As usual, $\delta$ denotes  Kronecker's Delta.

For the existence and uniqueness of $\mu$, see \cite{ZMobius}.
The M\"{o}bius function of a matroid is defined as  $\mu(M) = \mu_{L_M}(\bar{\emptyset} , E)$ where $\mu_{L_M}$ is the standard M\"{o}bius function on the lattice of flats and $\bar{\emptyset}$ is the least element in the lattice of flats (which contains all loops of the matroid). When $M$ has no loops the M\"{o}bius function of  $M$ is related to the Tutte polynomial via the equation $|\mu(M)| = T(M; 1,0)$ \cite{TutteP}.

\section{The Matroid Associated to the Action}

\indent Denote the $n$-torus by $T^r = T^1_1 \times T^1_2 \times \cdots T^1_r$.  We will also use the decomposition of an odd-dimensional sphere into circles: $S^{2n-1} = S^1_1 * S^1_2 * \cdots * S^1_n$, where * denotes the topological join of spaces.  In the interests of notational brevity, we will leave out the repeated superscript ``1" when referring to the circles in either decomposition. Given any linear action of $T^r$ on an even-dimensional sphere there is a pair of antipodal points which are fixed by the action.  Hence the quotient space is the suspension of a  linear action of $T^r$ on an odd-dimensional sphere.  As all of our questions of interest are easily answered for suspensions, we will henceforth assume that the sphere is odd-dimensional. 
   
We wish to study an effective linear actions $T^r \curvearrowright S^{2n-1}$ and the resulting quotient space $X = S^{2n-1}/T^r$. We  associate to each such action an $r \times n$ matrix $Z=(z_{ij})$ as follows:  Since $T^r$ consists of commuting $n \times n$ orthogonal matrices   we can simultaneously diagonalize all of the elements of $T^r$ over the complexes with diagonal entries in the unit circle.  Equivalently, $T^r$ is conjugate in $O(n)$  to a torus such that each $e^{\sqrt{-1}~\theta} \in T_i$ acts on $S_j$ by $e^{\sqrt{-1}~\theta} \cdot e ^{\sqrt{-1}~\beta} = e^{\sqrt{-1} (z_{ij} \theta + \beta)}, \ z_{ij} \in \mathbb{Z}. $   As conjugate tori give isometric quotient spaces, we will assume that $T^r$ is presented in this form.  

\begin{lemma} \label{matrix operations}  Performing any combination of the following $\mathbb{Z}$-matrix operations on $Z$ does not affect the isometry type of the corresponding quotient space.

\begin{enumerate}
\item \label{rowshuffle} Reordering the rows of $Z$ 
\item \label{columnshuffle} Reordering the columns of $Z$
\item \label{pm row} Multiplying any row by $\pm 1$ 
\item  \label{pm column} Multiplying a column by $\pm 1.$
\item  \label{row operation} Adding a multiple of one row to another

\end{enumerate}
\end{lemma}

\begin{proof} \hspace{1.5in}
          \\
\ref{rowshuffle}.) Switching two rows is equivalent to changing the order of the circles in the chosen basis $T^r = T_1 \times \cdots \times T_r$.\\
\ref{columnshuffle}.) Column switching is equivalent to  changing the order of the circles chosen in the join $S^{2n-1} = S_1* \cdots * S_n$.\\ 
\ref{pm row}.)  This corresponds to the choice of a preferred orientation for the circles of the torus.\\
\ref{pm column}.) This corresponds to the choice of a preferred orientation for the circles of the sphere.\\
\ref{row operation}.) Let $Z_i$ and $Z_j$ be rows of the matrix. If we replace $Z_j$ with $Z_j + cZ_i$, then the action $T^r \curvearrowright S^{2n-1}$ corresponding to the new matrix will be the action obtained by precomposing the original action with the group isomorphism $\phi: T^r \to T^r$ determined by the elementary matrix which is diagonal except for the $ji$ entry which is $c.$ \end{proof}

We  note that if $c \in \Z$ divides an entire row, say row $i,$  then the action is not effective as it has a kernel of the $c$-th roots of unity of $T_i.$   However, the quotient is isometric to the orbit space  of  $T^r/\Z_c\cong T^r$, where $\Z_c$ acts trivially except on $T_i$ .  We therefore allow division of an entire row in the matrix by $c$, provided that $c$ divides all of its entries.

As observed previously, there is a natural matroid associated to $Z$  which we denote by $M_Z.$  The ground set of $M_Z$ is the columns of $Z,$ and the independent subsets of $M_Z$ are the linearly independent subsets of columns.  An equivalent method for determining $M_Z$ is via representation theory.  The {\em real} irreducible representations of $S^1$ are isomorphic to $\Z/\pm1.$ So we can write the representation $\rho: T^r \to O(n)$ given by the action as a direct sum $\rho = \rho_1 \oplus \dots \oplus \rho_n$ where each $\rho_i \in (\Z)^r/ \pm 1.$ This means that $\{\rho_1, \dots, \rho_n\}$ has a matroid structure given by viewing each $\rho_i$ as a vector in $ \Q^r$ determined up to sign. It is not hard to see that this matroid is $M_Z.$  Thus $M_Z$ only depends on the action $T^r \curvearrowright S^{n-1},$ not on the chosen  diagonalization.    In fact, we will write $M_X$ for this matroid.  While it is not immediately obvious, it is possible to reconstruct $M_X$ from the isometry type of $X.$  As we have no use for this fact we do not prove it here.    
 While we prefer to use matroid notation for its simplicity, it is important to keep in mind that our matroids have  matrix representations  derived from the action.  Furthermore, when we refer to $M_X-e_j$ or $M_X/e_j$, we assume that there is a preferred class of representative matrices for these matroids.  In particular we will use $X_M$ to refer to  a quotient space even though the matroid (without a particular representation) does not determine the quotient space up to isometry. For instance, if $Z_1 = \begin{bmatrix} 2 & 3 \end{bmatrix}$ and $Z_2 = \begin{bmatrix} 1 & 1 \end{bmatrix},$ then the corresponding quotient spaces are non isometric two-spheres.  For an example where one quotient space is $\C P^n$ and the other is not even a manifold, see Section \ref{manifolds}.

\section{$H_\ast(X)$} \label{integral homology}

\subsection{$X$ as a mapping cone}
\indent Let $M_X$ be a matroid corresponding to a quotient space $X$.  Not surprisingly, it is possible to extract a variety of geometric and/or topological data from $X$ through the matroid structure of $M_X.$  

\begin{proposition} \label{loop} Let $X = S^{2n-1}/T^r$ and let $M_X$ be the corresponding matroid.  If $M_X$ contains a loop $e_j$, then $X = S_j * X_{M-e_j}.$ \end{proposition}

\noindent Note that $X = S_j * X_{M-e_j}$ means that $X$ is {\em isometric} to the given (spherical) join.  

\begin{proof}  If $e_j$ is a loop, then the $j^{th}$ column of any matrix representation of $M_X$ is the zero vector.  This implies that $T_i$ fixes $S_j$ for all $i$.  Since $S_j$ is fixed by the action of $T^r$, $X_M = S_j * X_{M-e_j}.$ \end{proof}

Now let us consider the situation when $e_j$ is not a loop.  
 Let $x$ be a point of $S_{j}$.  We will denote the stabilizer of $x$ in $T^r$ by $T^r_x$. We can decompose the quotient map on the sphere induced by the action of $T^r$ into two parts: $f: S^{2n-1} \twoheadrightarrow S^{2n-1}/T^r_x$ and $g: S^{2n-1}/T^r_x \twoheadrightarrow S^{2n-1} / T^r$.  Evidently $g$ is just the quotient map for the action of $T^r/T^r_x$ on $f(S^{2n-1})$.  Then $ g \circ f$ is the projection from $S^{2n-1}$ to $X$.  The entire circle $S_j$ is fixed by $f$, and  $g$ identifies all of $S_{j}$ to a single point $\bar{x}.$  Define $R_x$ to be the quotient of the action of $T^r$ restricted to the $(2n-3)$-dimensional sphere $(S_1 * \cdots * \hat{S_{j}} * \cdots * S_n)$.  As $T^r$ respects the join decomposition of $S^{2n-1}$ every point $\bar{y} \neq \bar{x}$ in $X$, but not in $R_x,$ lies on a unique minimal geodesic from $\bar{x}$ to $R_x.$  The minimal geodesics in $X$ with initial value $\bar{x}$ are parameterized by  $(S_1 * \cdots * \hat{S_{j}} * \cdots * S_n)/T^r_x.$  This quotient space is usually called the space of directions of $X$ at $x$ and we denote it by $N_x.$ All of the minimal geodesics from $\bar{x}$ to $R_x$ have length $\pi /2.$ The above discussion shows that $X$ is (homeomorphic to) the mapping cone of $g: N_x \to R_x$  with cone point $\bar{x}.$  As with any mapping cone, there is an associated Mayer-Vietoris sequence.
 
  \begin{equation} \label{les mv} \dots \to \tilde{H}_i (R_x) \to \tilde{H}_i (X) \stackrel{\partial}{\to} \tilde{H}_{i-1} (N_x) \to \dots .
 \end{equation}

 \begin{proposition} \label{coloop=cone} Let $X = S^{2n-1}/T^r$ and let $M_X$ be the corresponding matroid.  If $M_X$ contains a coloop, then $X$ is a cone. \end{proposition}

 \begin{proof} Let $e_j$ be a coloop of $M$.  Then we may row reduce the representative matrix of $M$ using the Euclidean algorithm so that $e_j$ is the $j$th column, this column contains only one nonzero entry, and that entry is in position $ij$.  In addition, the $i$-th row is zero except for $ij.$ Since the action is effective this entry must be plus or minus one.  With the matrix in this form, it is clear that $T^r_x = T^r$ for any $x \in S_j.$  Hence for this $x$ the map which determines the mapping cone structure of $X$ is the identity. \end{proof}
 
 The above results already make it easy to compute $\pi_1(X).$  If $n=1,$ then $X$ is homeomorphic to a circle or a point.  In all other cases, $X$ is simply connected.
 
 \begin{theorem} \label{simply connected}
 If $n \ge 2,$ then $X$ is simply connected.
 \end{theorem}
 
 \begin{proof}
 If $e_1 \in M_X$ is a loop or coloop, then Propositions \ref{loop} and \ref{coloop=cone} immediately imply $X$ is simply connected.  So assume that $e_1$ is neither a loop nor a coloop.  For the base case $n=2,$ the only remaining possibility is that $Z = [a_1 \ a_2]$ with both entries nonzero.  This implies $X$ is homeomorphic to $\C P^1$ and hence simply connected (see the proof of Proposition \ref{CPn}).  For the induction step, the mapping cone presentation of $X$ shows that $X$ is the union of two simply connected open subsets whose intersection is connected.  Apply Siefert-van Kampen.   \end{proof}

\subsection{$\tilde{\PP}(X,t)$}
In this section we prove that $H_i(X;\Z)$ is a free abelian group for all $i$ and that the integral reduced Poincar\'e polynomial 
$$\tilde{\PP}(X,t) = \displaystyle\sum \rk \tilde{H}_i(X, \Z) ~t^i$$
\noindent equals $t^{r-1} T(M_X; 0, t^2).$  Our strategy is to use induction on $n$, the  recursion which characterizes the Tutte polynomial, and the long exact sequence (\ref{les mv}).  If $M_X$ contains a coloop, then Proposition \ref{coloop=cone} works well.  However, if $M_X$ does not contain a coloop, then an immediate obstacle to induction is that $N_x$ may not be a quotient of a sphere by a real torus.  

Let $x \in S_j.$ Recall that $N_x \cong S^{2n-3}/T^r_x$, so we wish to better understand the structure of this isotropy subgroup. By Lemma \ref{matrix operations} we can use the Euclidean algorithm to  row reduce a representative matrix of $M_X$ so that there is only one nonzero entry in column $j$, let us say it is in row $i$.  If this $ij^{th}$ entry is a one, then $T^r_x \cong T_1\times \cdots \times \hat{T}_j \times \cdots \times T_r$.  If the entry is some $a\neq 1$, then $a$ is the gcd of column $j$. Hence, $T^r_x \cong T_1\times \cdots \times \hat{T}_j \times \cdots \times T_r \times \ \Z_a$, where $\Z_a$ is the cyclic group $\Z/a\Z.$  This demonstrates that $N_x \cong S^{2n-3}/T^r_x$ where $T^r_x \cong T^{r-1} \times \Z_a$ for some $a \in \mathbb{N}$.  We can break up this action into two parts: let $\hat{N_x} \cong S^{2n-3}/T^{r-1}$ so that $N_x = \hat{N}_x/\Z_a$ and   the matroid corresponding to  $\hat{N_x}$ is $M_X/e_j$. 
\\

We wish to show that this extra quotient by a finite group does not affect the rational homology of $\hat{N}_x$.  In order to do so, we require more information about the local structure of $N_x$ and $\hat{N}_x.$\\

An \emph{absolute neighborhood retract} (ANR) is a topological space $Y$ with the property that for every normal space $Z$ that embeds in $Y$ as a closed subset, there exists an open set $U$ in $Y$ such that $Z \subset U \subset Y$. 

\begin{lemma} $N_x$ and $\hat{N}_x$ are both ANRs.
\end{lemma} 

\begin{proof}
We say that an action has \emph{finite type} if there are only a finite number of conjugacy classes of isotropy subgroups. It is shown in Conner \cite{Conner} that if $\Gamma$ is a compact abelian Lie group acting on a compact connected finite dimensional ANR $X$, and the action is of finite type, then the orbit space $X/\Gamma$ is an ANR.  This result also applies to all finite abelian groups $\Gamma$.  It is well known that every sphere is an ANR.  It remains to be shown that the linear action of $T^r$ on $S^{2n-1}$ has finite type.  By the definition of the action, all the points $x$ on any given invariant circle $S_j$ have the same isotropy group $T^r_x$.  If $x \in S^{2n-1}$ does not lie on an invariant circle, then there is some minimal subset of circles $\{S_{i_k}\}_{k=1}^m$ whose join in $S^{2n-1}$ contains $x$.        By choosing points  $y_{i_k}\in S_{i_k}$, we see that  $T^r_x = \displaystyle \bigcap T^r_{y_{i_k}}$.  This formulation demonstrates that the toral action can only have a finite number of distinct isotropy groups and is thus of finite type.
\end{proof}

\begin{lemma} \label{Bre}
Suppose a finite abelian group $G$ acts on $\hat{N}_x$.   Let $F$ be a field of characteristic 0 or of characteristic prime to the order of $G$.  Then $H_n(\hat{N}_x/G; F) \cong [H_n(\hat{N}_x; F)]^G$, the group of invariant homology classes.
\end{lemma}

\begin{proof}
For \v{C}ech cohomology, the lemma is a corollary of Theorem III.7.2 in Bredon's text on transformation groups  \cite{Bredon} which states the result for more general quotient spaces.  By the previous lemma, $N_x$ and $\hat{N}_x$ are both ANRs.  The lemma follows directly since singular cohomology and \v{C}ech cohomology are equivalent on ANRs and we are working with field coefficients.
\end{proof}

\begin{proposition} \label{Nx}
$H_\ast(N_x ; \Z) \cong H_\ast(\hat{N_x}; \Z).$
\end{proposition}

\begin{proof}
Let $\kk$ be a field.  If the characteristic of $\kk$ is zero, then choose any column of $Z$.  When the characteristic of $\kk$ is positive, choose a column $j$ so that the characteristic of $\kk$ does not divide the gcd of the entries of the column.  There is always such a column, otherwise the action would not be effective.  Write $N_x = \hat{N}_x/\Z_{a_j}$ as above. The finite group $\Z_{a_j}$ is a subgroup of the {\em connected} group $T_j$ which acts on $N_x$ by isometries.  Hence every element of $\Z_{a_j}$ acts on $N_x$ by a map homotopic to the identity.  Now, Lemma \ref{Bre} shows that for any field $\kk, \ H_\ast(N_x ; \kk) \cong H_\ast(\hat{N_x}; \kk).$  The universal coefficient theorem finishes the proof.
\end{proof}

With the main obstacle to induction out of the way we are ready to prove the main theorem of this section.

\begin{theorem} \label{poincarepoly} 
Let $X = S^{2n-1}/T^r$ be a quotient of an odd-dimensional sphere by an effective linear action.  Then $H_\ast(X;\Z)$ is a finitely generated abelian group and 
\begin{equation} \label{main equation} \tilde{\PP}(X,t) = \displaystyle\sum \rk \tilde{H}_i(X, \Z) ~t^i = t^{r-1}~T(M_X; 0, t^2).
\end{equation}
 \end{theorem}

\begin{proof}
It is sufficient to prove (\ref{main equation}) when using arbitrary field coefficients.  So let $\kk$ be a field (of any characteristic). 

We proceed by induction on $n$.  When $n$ is one there are only two actions to consider.  The circle acting on itself and the trivial action of $T^0 = \{id\}$ on the circle. The latter is an effective action in the sense that every nonidentity element of the group acts nontrivially!  In both cases (\ref{main equation}) is easily verified. 

For the induction step there are three cases to consider: $e_1 \in M_X$ is a coloop, loop, or neither.   If $e_1$ is a coloop, then Proposition \ref{coloop=cone} tells us that $X$ is contractible, so $\tilde{\PP}(X,t) = 0$,  while Tutte recursion insures that $T(M_X; 0, t^2) = 0.$ When $e_1$ is a loop, Proposition \ref{loop} implies that $X = S^1 \ast X_{M_X- e_1}.$  So the induction hypothesis insures that $\tilde{\PP}(X,t) = t^2 ~\tilde{\PP}(X_{M_X-e_1}) = t^{r-1} ~t^2 ~T(M_X-e_1; 0,t^2)  = t^{r-1}~T(M_X; 0, t^2).$

So assume that $e_1$ is neither a loop nor a coloop. Then we have that $r(M-e_1) = r(M)$ and $r(M_X/e_1) = r(M_X) - 1$.   Now consider the long exact sequence (\ref{les mv}).

$$ \dots  \tilde{H}_i(N_x; \kk) \to \tilde{H}_i(R_x; \kk) \to \tilde{H}_i (X; \kk) \stackrel{\partial}{\to} \tilde{H}_{i-1} (N_x; \kk) \to \tilde{H}_{i-1}(R_x; \kk) \to \dots$$
\noindent  The induction hypothesis applied to $R_x$ and $N_x$ (via Proposition \ref{Nx}) implies that, depending on the parity of $i,$ one of two things is happening.  One,   $\tilde{H}_i(R_x; \kk)=0$ and $\tilde{H}_{i-1}(N_x; \kk)=0,$ in which  case $\tilde{H}_i(X; \kk) = 0.$  Or, $\tilde{H}_i(N_x; \kk)=0$ and $\tilde{H}_{i-1}(R_x; \kk) =0$, in which case $\tilde{H}_i (X; \kk) \cong \tilde{H}_i(R_x; \kk) \oplus \tilde{H}_{i-1}(N_x; \kk).$ Combining these two possibilities with the induction hypothesis gives
$$\tilde{\PP}(X,t) = \tilde{\PP}(R_x,t) + t~\tilde{\PP}(N_x, t) = t^{r-1} T(M_X-e_1; 0, t^2) + t^{r-1} T(M_X/e_1; 0,t^2)$$
$$ = t^{r-1} T(M_X; 0, t^2).$$

\end{proof}

The above formula raises two immediate questions. Since all of the homology groups are finitely generated free abelian, $H^i(X) \cong H_i(X)$ for every $i.$  

\begin{prob}
What is the ring structure of $H^\ast(X)?$
\end{prob}
\noindent If the dimension of $X$ is odd, or the rank of $M_X$ is greater than $n,$ then all products in $H^\ast(X)$ must be trivial for purely dimensional reasons. When $M_X$ is rank one without loops $X$ is a weighted projective space. (See Section \ref{manifolds}.) In that special case the ring structure of the cohomology ring of $X$ was determined in \cite{Kawasaki}.   We do not know of any other case with nontrivial products.  

As the homology of $X$ vanishes in every other degree, it is natural to ask whether or not the following holds.
\begin{prob}
Is there a CW-decomposition of $X$ so that all boundary maps are zero?
\end{prob}
\noindent  If so, one might hope that Tutte's theory of basis activity for graphs \cite{Tutte}, extended to matroids by Crapo \cite{Crapo}, might be realized with a natural bijection between the cells of the CW-structure and the bases of $M_X$ with internal activity zero.  

\section{The Singular Set} \label{singset}

Given a quotient space $X = Y/T^r$, the \emph{rational singular set} of the action is the image in the quotient space of the points of $Y$ whose isotropy subgroups are infinite subgroups of $T^r$.  We will denote the rational singular set of the quotient $S^{2n-1}/T^r$ by $\mathcal{S}$ and determine its homotopy type.  

Let $A = \{e_{i_1}, \dots, e_{i_k}\} \subseteq M_X.$  Define $S^A =\{(x_1, x_2, \dots, x_{2n-1}) \in S^{2n-1}: x_{2i-1}=x_{2i}=0 \mbox{ for all } e_i \notin A.\}.$ Equivalently, $S^A$ is the join $S_{i_1} \ast \cdots \ast S_{i_k}$ in  $S^{2n-1}.$  For $x \in S^{2n-1}$ set $A_x$ to be the minimal $A$ such that $x \in S^A.$

  Any $t \in T^r_x$ must fix all of $S^{A_x}.$ Suppose $A_x$ is a spanning subset of $M_X.$   Then the square submatrix of $Z$ whose columns are $A_x$ can be diagonalized over the integers with nonzero diagonal entries $\{c_1, \dots, c_r\}$ by the elementary row operations covered by Lemma \ref{matrix operations}.  This implies $T^r_x$ is contained in a subgroup of $T^r$ isomorphic to $\Z_{c_1} \oplus \cdots \oplus \Z_{c_r}$ and hence is finite.    So, for any $x \in \SSS$ we see that $A_x$ is  a nonspanning subset and hence contained in a hyperplane $H$ of $M_X.$

 Conversely, suppose $H$ is  a hyperplane of $M_X$. In the column space of $Z$, $H$ is the intersection of  the columns of $Z$ with a rational hyperplane.  This hyperplane is perpendicular to an integer vector. Thus there is an element $\gamma$ of the row space which is an integral linear combination of the rows of $Z$ such that the zeros of $\gamma$ correspond to the columns in $H.$  Therefore, there is an element of $T^r$ of infinite order which fixes $S^H.$  As a result, we now know that the preimage of the rationally singular set is the union of all $S^H, ~H$ a hyperplane of $M_X.$ Each $S^H$ is a sphere of dimension $2|H|-1$. Hence the image of $S^H$ in $X$ is of dimension $2|H|-1 - r(H) = 2|H| - 1 - (r(M)-1) = 2|H|-r(M)$.

We define an \emph{arrangement} as a finite collection $\mathcal{A} = \{A_1, \dots, A_m\}$  of closed subspaces of a topological space $U$ such that: \\
i) $A, B \in \mathcal{A}$ implies that $A\cap B$ is a union of spaces in $\mathcal{A}$\\
ii) If $A, B \in \mathcal{A}$ and $A \subseteq B$, then the inclusion map $A \hookrightarrow B$ is a cofribration.
\\

Given $A \subseteq M_X$ define $X_A$ to be $g \circ f(S^A).$  
  Let $\mathcal{A}$ be the set generated by $\{X_H: H \mbox{ is a hyperplane of } M_X\}$ and all of its intersections, including the empty set if this is the intersection of all the $g \circ f(H).$
Now let $P$ be the poset whose elements are the sets in $\mathcal{A}$, ordered by reverse inclusion. The $X_H$ in $\SSS$ are the minimal elements of $P.$ 
Furthermore, the elements of $P$ corresponds to flats of the matroid $M_X$.  In fact, P is isomorphic to $(L_{M_X})^*$, the {\em order} dual of the lattice of flats of M with the maximal element $\hat{1}$ corresponding to $M_X$ removed. In other words,  $P$ is the poset of flats of $M_X$, other than $M_X$, ordered by reverse inclusion.

\begin{proposition} If $X_F, X_G$ are elements of the arrangement $\mathcal{A}$ and $X_F > X_G,$ then the inclusion map $X_F \hookrightarrow X_G$ is homotopic to the constant map.\end{proposition}

\begin{proof}

   Let $c \in  S^{2n-1}$ be a point on an invariant circle $S_j$ such that $e_j$ is in the flat $G, $ but not $F.$  By Proposition \ref{coloop=cone} $X_{F \cup \{e_j\}} \subseteq X_G$ in $X$ is a cone with base $X_F$ in $X_G.$  

\end{proof}

 The above proposition means that we can use the wedge lemma from \cite{ZZTop} to compute the homotopy type of $\SSS$.  
  
\begin{theorem}  \label{sing homotopy} 
The singular set $\SSS$ is homotopy equivalent to 
$$\bigvee_{\stackrel{F \in L_{M_X}}{F \neq E}} X_F \ast \vee^{\mu(M_X/F)}_{i=1} S^{r - r(F) -2}.$$
\end{theorem}

\begin{proof}
By the wedge lemma in \cite{ZZTop}, $\SSS$ is homotopy equivalent to 
$$\bigvee_{X_F \in P} X_F \ast \Delta(P_{< X_F}),$$
where $P_{< X_F}$ is the subposet of $P$ consisting of all elements of $P$ strictly less than $X_F.$  By definition this is the order dual of the interval $[F,E]$ in $L_{M_X}$ with $F$ and $E$ removed.  Since the order complex of a poset and its order dual are isomorphic, the result now follows from the fact that $[F,E] \cong L(M_X/F)$ and Theorem \ref{ordercomplex}.
\end{proof}

 With the homotopy type of singular set in hand, it is easy to compute the reduced Poincar\'e polynomial of $\SSS.$
  
\begin{theorem} The reduced Poincare polynomial of the singular set of the action with integral coefficients is given by $\tilde{\mathbb{P}}(\mathcal{S}, t) = t^{r(M)-2} [  T(M; 1, t^2) - T(M; 0, t^2)]$.\end{theorem}

\begin{proof}
  By the previous theorem, Theorem \ref{poincarepoly} and the results of Section \ref{integral homology}, 
$$\tilde{\PP}(\SSS, t) = \displaystyle\sum_{\stackrel{F \in L_{M_X}}{F \neq E}} \tilde{\PP}(X_F  \ast \vee^{\mu(M_X/F)}_{i=1} S^{r-r(F)-2},t)$$
\begin{equation} \label{singpoly}
= \displaystyle\sum_{\stackrel{F \in L_{M_X}}{F \neq E}} t^{r-2}~\mu(M_X/F)~  T(F; 0,t^2), 
\end{equation}
\noindent The last equality uses the usual computation of $\tilde{\PP}$ for the join of a space and a wedge of spheres of the same dimension via the K\"unneth theorem.  

Recall from the end of Section \ref{tutte invariants} that $\mu(M_X/F) = T(M_X/F; 1,0)$ whenever $M_X$ has no loops. This is the case here, since for any flat $F$ of any matroid $M$, $M/F$ has no loops.  Now compare (\ref{singpoly}) with the Kook, Reiner Stanton convolution formula \cite{Kook} for any matroid $M,$
$$T(M; 1,t^2) = \displaystyle\sum_{F \in L_M} T(M/F; 1,0) \cdot T(F; 0,t^2).$$
\end{proof}

\section{Manifolds} \label{manifolds}

One natural question to ask is, ``When is $X$ a topological manifold?"  The Tutte polynomial specialization $T(M_X; 0,t^2)$ is always of the form
$$ t^{2(n-r)} + b_{n-r-1} t^{2(n-r-1)} + \dots + b_1 t^2,$$
with $b_i \in \Z_{\ge 0}.$ Furthermore, if $b_i > 0$ and $i<n,$ then $b_{i+1} >0$.  This can be shown by a deletion and contraction argument.   This fact, Poincar\'e duality, and our formula for $H_\ast(X;\Q),$ imply that there are only two potential situations for $X$ to be a manifold (without boundary). If $r =1$, then $M_X$ is the rank one uniform matroid, where a subset of $M_X$  is independent if and only if it has cardinality one or zero.  Otherwise, we have $T(M_X; 0, t^2) = t^{2(n-r)}.$  If $M_X$ does not meet one of these two criteria, then the $(2n-1-r)$-dimensional $X$ has $H_{2n-r-3}(X, \Q) \neq 0$ and $ H_2(X, \Q) =0$, preventing the orbit space from satisfying Poincar\'e duality.

\subsection{$r=1$}

In this case $Z=[a_1 a_2 \dots a_n]$ with all $a_i \neq 0.$  These quotient spaces have been studied under the names twisted projective spaces or weighted projective spaces.  By using Lemma \ref{matrix operations}  we can further simplify and assume all the $a_i$ are positive and $a_1 \ge a_2 \ge \dots \ge a_n$. For instance, if $Z=\begin{bmatrix} 1 & 1  &\dots& 1\end{bmatrix},$ then $X$ is $\C P^n.$ As we will see  below, when $A=\begin{bmatrix} a_1 & a_2 \end{bmatrix}$, the quotient $X$ is always homeomorphic to $S^2$.   On the other hand, consider $Z=\begin{bmatrix} 3 & 1  &\dots& 1 \end{bmatrix}.$  Let $x \in S_1.$  Then $T_x \cong \Z_3$ which acts on all $S_j, j \neq 1,$ by rotation.  Hence $N_x$ is a lens space and excision shows $H_\ast(X, X-\{\bar{x}\})$ is not isomorphic to the homology of a sphere.  Thus $X$ cannot be a manifold. 

The necessity portion of Proposition \ref{CPn} can easily be obtained by applying \cite[Theorem 1]{Kawasaki}.  However, we give a direct proof here.  

\begin{proposition}  \label{CPn}
If $Z=\begin{bmatrix}a_1 & \dots & a_n \end{bmatrix}$ with all $a_1 \ge a_2 \ge \dots \ge a_n  > 0$, then $X$ is a manifold if and only if $n=2,$  or $a_1 = a_2 = \dots = a_{n-1}$ and $a_n = 1.$ Furthermore, if $X$ is a manifold, then $X$ is homeomorphic to $\C P ^n.$
 \end{proposition}

\begin{proof}
Suppose $n=2$.  Consider the mapping cone structure of $X.$  Since $a_2 \neq 0, \ R_x$ is a point.  On the other hand, $N_x = S^1/\Z_{a_1}$ and hence homeomorphic to the circle.  So $X$ is homeomorphic to the mapping cone of the circle mapped to a point.  Thus $X$ is homeomorphic to $\C P^1.$

Now we assume $n \ge 3.$  Let $x \in S_1.$  Then $T_x = \Z_{a_1}$ and $T_x$ acts nontrivially on every circle $S_j$ with $a_j < a_i.$  The homology of $N_x,$ and by excision, the pair $(X, X-\{\bar{x}\}),$ will not be that of a sphere unless the action of $T_x$ is trivial on all the $S_i, i \neq n$ \cite{Willson}.  Thus $a_1 = a_2 = \dots = a_{n-1}.$  To see that $a_n=1,$ suppose $a_n > 1.$  Let $z \in S_n,$ so $T_z = \Z_{a_n}.$  If $a_n$ divides all of the other $a_i,$ then the action is not effective.  If it does not, the $T_z$ acts nontrivially on all of the other circles and the usual excision argument shows that $X$ cannot be a manifold.  

Lastly, we have to show that if $Z=\begin{bmatrix} a_1 & \dots & a_1 & 1 \end{bmatrix},$ then $X$ is homeomorphic to $\C P^n.$  Break up the action $T \curvearrowright S^{2n-1}$ into two parts.  First quotient out by $\Z_{a_1}.$  This subgroup acts trivially on all the circles except $S_n.$  Hence it leaves a quotient space $\bar{X}$ homeomorphic to $S^{2n-1}.$  Now act on $\bar{X}$ by $T/T_{a_1}.$  This group is also a rank one torus and this action is equivalent to $\bar{Z} = \begin{bmatrix} 1 & \dots & 1 \end{bmatrix}.$  Therefore, $X$ is homeomorphic to $\C P^n.$
\end{proof}

\subsection{Spheres}

When is $T(M_X; 0, t^2) = t^{2(n-r)}?$  Obviously this is the same as when $T(M_X; 0,t) = t^{n-r}.$ 

\begin{proposition}  \label{tp-circuits}
Let $M$ be a rank $r$ matroid with $n$ elements.  Then $T(M;0,t) = t^{n-r}$ if and only if $M$ is a direct sum of circuits. 
\end{proposition}

\begin{proof}
An elementary deletion-contraction argument using the recursive definition of the Tutte polynomial.
\end{proof}

As noted in the introduction, one of the obvious questions when considering linear quotients of spheres is, ``When is the quotient space homeomorphic to a sphere?"  For real tori the answer is, at least in the language of matroids,  essentially the same as for $\Z_2$-tori \cite[Theorem 4]{BSQ}.  In preparation for this result we consider what happens when $M_X$ is a direct sum of smaller matroids.  

Suppose $M_X = M_1 \bigoplus \cdots \bigoplus M_l.$ Let $n_i$ be the cardinality of $M_i,$  and $r_i$ the rank of $M_i.$ So $\sum n_i = n$ and $\sum r_i =r.$ Then it is possible, after applying Lemma \ref{matrix operations}, to write $Z$ in block diagonal form with $l$ blocks of size $r_i \times n_i.$  Denote the blocks by $Z_i.$  Each $Z_i$ corresponds to a quotient space $X_i = S^{2n_i -1}/ T^{r_i}.$  Now it is possible to write $T^r = T^{r_1} \times \cdots \times T^{r_l},$ and  $S^{2n-1} = S^{2n_1-2} \ast \cdots \ast S^{2n_l-1}$ so that each $T^{r_i}$ acts trivially on every $S^{2n_j - 1}$ when $i \neq j.$ From these decompositions the following proposition is clear. 

\begin{proposition} \label{direct sum=join}
Suppose $M_X = M_1 \bigoplus \cdots \bigoplus M_l$  with notation as above.  Then $X = X_1 \ast \cdots \ast X_l.$
\end{proposition}

\begin{theorem}
The following are equivalent.
\begin{enumerate}
   \item  \label{circ}
  $M_X$ is a direct sum of circuits.
  \item  \label{homeo. sphere}
  $X$ is homeomorphic to a sphere.
  \item  \label{homol. sphere}
  $X$ is an integral homology sphere.
  
\end{enumerate}
\end{theorem}

\begin{proof}
Obviously (\ref{homeo. sphere}) implies (\ref{homol. sphere}).  The implication (\ref{homol. sphere}) implies (\ref{circ}) follows immediately from Proposition \ref{tp-circuits} and our formula for the homology of $X.$  So it remains to prove (\ref{circ}) implies (\ref{homeo. sphere}).   Our first simplification is to observe that since joins of spheres are spheres, Proposition \ref{direct sum=join} shows that it is sufficient to prove that if $M_X$ is a circuit, then $X$ is homeomorphic to a sphere.   If $n=1,$ then $r=0$ and $X = S^1.$ So from here on we assume that $r+1=n \ge 2$ and $M_X$ is a circuit.  This implies that $Z$ can be row reduced to something of the form
$$\begin{bmatrix} a_1 & 0 & 0 &\dots & 0 & b_1 \\
0 & a_2 & 0 & \dots & 0 & b_2\\
\vdots & \vdots & \vdots & \vdots & \vdots & \vdots \\
0 & 0 & 0 & \dots & a_{n-1} & b_{n-1} \end{bmatrix}$$
with all $a_i$ and $b_i$ nonzero. 

Our strategy here is simple:  prove that $X$ is a simply connected manifold with the homology of a sphere.  We have already seen that $X$ is simply connected (Theorem \ref{simply connected}) and that it has the homology of a sphere.  So it remains to show that $X$ is a manifold.  We will do this by induction on $n.$  The base case $n=2$ was discussed in the proof of Proposition \ref{CPn}. 

Let $\bar{x} \in X$ and $x$ be any preimage  of $\bar{x}$ in $S^{2n-1}.$  Now let $N_x$ be the unit tangent vectors in the tangent space of $x$ which are orthogonal to $T^r x,$ the orbit of $x.$  Since small metric neighborhood of $\bar{x}$ are homeomorphic to a cone over $N_x/T^r_x,$ it is sufficient to prove that this quotient space is homeomorphic to $S^{n-1}.$  

As before, for $x \in S^{2n-1}$ let $A_x$ be the  minimal nonempty subset of $M_X$ such that $x \in S^{A_x}.$ If $A_x= M_X,$ then $x$ is in the principle isotropy group of the torus and $\bar{x}$ is a manifold point.  So we can assume that $A_x \neq M_X.$  For notational convenience, we can also assume that $e_n \notin A_x$ by reordering the columns if necessary.  Since $T^r \cdot S^{A_x} \subseteq S^{A_x}, \  T^r x \subseteq S^{A_x}.$ Define three subspaces of the tangent space of $x$ as follows:  $\mathcal{T}$ are the vectors tangent to $T^r x, \ \mathcal{O}$ are vectors tangent to $S^{A_x},$ but orthogonal to $T^r x,$ and $\mathcal{N}$ are those orthogonal to $S^{A_x}$, and thus also orthogonal to $T^r_x.$    Then the tangent space at $x$ is  $\mathcal{T} \oplus \mathcal{O} \oplus \mathcal{N}.$ In terms of this decomposition, $N_x$ are the unit vectors in $\mathcal{O} \oplus \mathcal{N}.$  The form of $Z$ implies that the rank $T^r_x$ is $|A_x|,$ unless $A_x=M_X,$ in which case it is only $n-1=r.$  By construction $T^r_x$ acts trivially on $\mathcal{O}.$ Thus we have reduced the problem to showing that $\tilde{N}_x = S^{M_X-A_x}/T^r_x$ is homeomorphic to a sphere, where $\tilde{N}$ are the unit vectors in $\mathcal{N}.$  The action of $T^r_x$ on $S^{M_X-A_x}$ is the induced action of $T^r$.  This is most easily seen by representing the unit vectors by minimal geodesics beginning at $x$ and ending in $S^{M_X-A_x}.$  

As we have seen, $T^r_x$ is a direct sum of a finite group and a torus.

$$T^r_x = \bigoplus_{i \in A_X} \Z_{a_i} \oplus \bigoplus_{i \notin A_x} T_i.$$
First we quotient out by the finite left-hand summand.  The action of each cyclic subgroup in this summand on $S^{2n-1}$ is trivial except for $S_n$ where it acts by rotation by $2\pi/ a_i.$ Hence, after quotienting out by the finite left-hand summand we are left with $S^{2n-1}$ with the same join decomposition as before (except $S_n$ is smaller)  and an action of $\tilde{T}^{n-|A_x|}$ whose associated matrix is the same form as before, but of smaller size and all of the $a_i=1.$  Finally, we can apply the inductive hypothesis to see that this quotient space is homeomorphic to a sphere. 

\end{proof}
\bibliographystyle{amsplain}
\bibliography{Bibliography2}

\end{document}